\documentclass[10pt]{article}

\usepackage{amsmath}
\usepackage{amsfonts}
\usepackage{amsthm}
\usepackage{thm-restate}
\usepackage{enumerate}
\usepackage{framed}
\usepackage{graphicx}

\newtheorem{oldthm}{Theorem}
\newtheorem{thm}{Theorem}
\newtheorem{lem}[thm]{Lemma}
\newtheorem{prop}[thm]{Proposition}
\newtheorem{cor}[thm]{Corollary}
\newtheorem{conj}[thm]{Conjecture}
\newtheorem{oldqn}[oldthm]{Question}
\newtheorem{qn}[thm]{Question}

\theoremstyle{definition}
\newtheorem*{defn}{Definition}

\newcommand{\sset}[2]{[#1]^{(#2)}}

\newcommand{\A}{\mathcal{A}}
\newcommand{\F}{\mathcal{F}}
\newcommand{\N}{\mathbb{N}}
\newcommand{\Pb}{\mathbb{P}}
\newcommand{\pg}{+}
\newcommand{\tup}[1]{^{(#1)}}
\newcommand{\otup}[1]{^{\underline{#1}}}

\newcommand{\es}{\emptyset}
\newcommand{\sm}{\setminus}

\newcommand{\ith}[1]{$#1$-th}
\newcommand{\dd}{d^\ast}

\newcommand{\nandr}{Ne\v{s}et\v{r}il and R\"{o}dl}
\newcommand{\candh}{Caro and Hansberg}

\begin{document}

\nocite{*}

\title{Degrees in oriented hypergraphs and sparse Ramsey theory}

\author{Vytautas Gruslys\footnote
{
Department of Pure Mathematics and Mathematical Statistics,
Centre for Mathematical Sciences,
Wilberforce Road,
Cambridge CB3 0WB,
United Kingdom;
e-mail: {\tt v.gruslys@dpmms.cam.ac.uk}
}}

\maketitle

\begin{abstract}
    Let $G$ be an $r$-uniform hypergraph. When is it possible to orient the
    edges of $G$ in such a way that every $p$-set of vertices has some
    $p$-degree equal to $0$? (The $p$-degrees generalise for sets of vertices
    what in-degree and out-degree are for single vertices in directed graphs.)
    \candh{} asked if the obvious Hall-type necessary condition is also
    sufficient.

    Our main aim is to show that this is true for $r$ large (for given $p$),
    but false in general. Our counterexample is based on a new
    technique in sparse Ramsey theory that may be of independent interest.
\end{abstract}

%\begin{keyword}
%    oriented hypergraphs \sep amalgamation \sep sparse Ramsey theory
%\end{keyword}

\section{Introduction}

When does a graph $G$ have an orientation such that every out-degree is at most $k$?
An obvious necessary condition is that $|E(H)| \le k|V(H)|$ for every subgraph $H \subset G$.
Indeed, suppose $G$ has such an orientation and $H \subset G$. Inside $H$ the sum of
out-degrees equals the number of edges. Moreover, each vertex contributes at most
$k$ to this sum, hence the condition. Hakimi proved that this condition is in fact sufficient.

\begin{oldthm}[Hakimi, \cite{hakimi65}]
	\label{thm:hakimi}
	Let $G$ be a graph and $k \ge 0$ an integer. Then $G$ has an orientation such that
	every vertex has out-degree at most $k$ if and only if all subgraphs $H \subset G$
	satisfy $|E(H)| \le k|V(H)|$.
\end{oldthm}

In fact, Hakimi proved a slightly stronger result that determines all possible out-degree
sequences produced by orientations of a given graph $G$. His proof uses induction on the number
of edges, but the weaker statement given above is a straightforward consequence of Hall's marriage
theorem. 

What about hypergraphs? Suppose an $r$-uniform hypergraph $G$ (i.e. a family of $r$-sets)
is given an \emph{orientation}, by which we mean that for each edge $e$ one
ordering of the vertices of $e$ is chosen.  This ordering is called the
\emph{orientation of $e$}. Note that each edge has exactly $r!$ possible orientations. If $r = 2$
then this coincides with the usual definition of graph orientation. We will often denote an orientation
of $G$ by $D(G)$ and the corresponding orientation of an edge $e$ by $D(e)$. 

Given an orientation $D(G)$, a vertex $v$ and $i \in [r] = \left\{1, 2, \dots, r\right\}$, the \emph{$i$-degree}
of $v$, written $d_i(v)$, is the number of edges $e$ such that $v$ is in the \ith{i} position of
$D(e)$. Note that if $r = 2$ then $d_1(v)$ is  the out-degree and $d_2(v)$ is the in-degree of $v$.

When does an $r$-uniform hypergraph $G$ have an orientation such that $d_1(v) \le k$
for every vertex $v$? Again, an obvious necessary condition is that $|E(H)| \le k|V(H)|$
for every subgraph $H \subset G$ (where as usual $H$ is a subgraph of $G$ if $V(H) \subset V(G)$ and
$E(H) \subset E(G)$). Indeed, inside $H$ the sum of $d_1(v)$ over all vertices $v$ of $H$
is equal to the number of edges of $H$, and is at most $k$ times the number of vertices.

Caro and Hansberg showed that this condition is sufficient.

\begin{oldthm}[\candh, \cite{caro12}] \label{thm:hakforhyp}
	Let $G$ be an $r$-uniform hypergraph and $k \ge 0$ an integer. Then $G$ has an orientation
	such that $d_1(v) \le k$ for all vertices $v$ if and only if all subgraphs
	$H \subset G$ satisfy $|E(H)| \le k|V(H)|$.
\end{oldthm} 

They proved it by constructing a suitable maximal flow on $H$, and a simple proof via
Hall's marriage theorem is also possible.

Now, in contrast to the situation for graphs, for oriented hypergraphs there is a sensible notion
of degree for sets of \emph{multiple} vertices. For example, given an orientation $D(G)$
and a pair of vertices $u,v$, we can define $d_{12}(u,v)$ to be the number of edges
$e$ such that $u$ and $v$ (in some order) are in the first two positions of $D(e)$.
So if the oriented edges are $(4,5,1), (4,1,3), (1,4,2)$ (where the vertex set is $[5]$) then
$d_{12}(1,4) = 2$.

More generally, for a $p$-set of vertices $A = \left\{v_1,\dotsc,v_p
\right\} \subset V$ and a $p$-set $I \subset [r]$, the \emph{$I$-degree}
of $A$, denoted by $d_I(A)$, is the number of edges $e$ such that the elements
of $D(e)$ in positions labeled by $I$ are $v_1, \dotsc, v_p$ (in some order).
More formally, $d_I(A)$ is the number of edges $e$ such that if we write $D(e) = 
\left(x_1, \dotsc, x_r \right)$ then $\{x_i \colon i\in I\}$ is exactly the set $A$.

[We mention in passing that there is a  variant of this notion where the mutual order of
$u$ and $v$ is important.  However, this alternative definition turns out to be less interesting for
the types of questions examined by \candh{} and by us. In this paper we mainly consider the former
notion (of `unordered' degrees), but  a brief analysis of the latter is given in Section
\ref{sec:ordereddegrees}.]

In \cite{caro12} \candh{} asked if a result similar to their Theorem \ref{thm:hakforhyp} can be found in
the setting of degrees for multiple vertices.

\begin{oldqn}[\candh, \cite{caro12}]
	\label{qn:genhak}
	Fix integers $k \ge 0$ and $1 \le p \le r$. Which $r$-uniform hypergraphs $G$ have
    an orientation such that $d_{[p]}(A) \le k$ for all $p$-sets of vertices $A \subset V$?
\end{oldqn}

Again there is an obvious necessary condition: if $G$ has such an orientation then for each
collection of $p$-sets $U \subset V\tup{p}$ at most $k|U|$ edges $e$ satisfy $e\tup{p} \subset U$
(here we write $X\tup{p}$ for the set of all $p$-subsets of $X$).  Indeed, given $U$, every $e$
with $e\tup{p} \subset U$ contributes $1$ to the sum $\sum_{A \in U} d_{[p]}(A)$, and this sum is
at most $k|U|$.

Our first aim is to answer the question of \candh{} by showing that this condition is sufficient.
Note that case $p = 1$ is Theorem \ref{thm:hakforhyp}.

\begin{restatable}{thm}{hakforpsets} \label{thm:hakforpsets}
	Let $G$ be an $r$-uniform hypergraph and $k \ge 0, 1 \le p \le r$ integers. Then $G$ has
    an orientation such that $d_{[p]}(A) \le k$ for all $A \in V\tup{p}$ if and only if for each
    $U \subset V\tup{p}$ at most $k|U|$ edges $e$ satisfy $e\tup{p} \subset U$.
\end{restatable}

\candh{}'s interest in Theorem \ref{thm:hakforhyp} was to answer the following question. When does
an $r$-uniform hypergraph $G$ have an orientation such that for every vertex $v$ there is \emph{some}
$i \in [r]$ such that $d_i(v) \le k$? It is once again easy to obtain a necessary condition:
such an orientation would partition the vertices into sets $V_1, \dotsc, V_r$ where each $v$
is put into some $V_i$ with $d_i(v) \le k$. Theorem \ref{thm:hakforhyp} applied to the induced
subgraphs $G[V_i]$ gives the necessary condition $|E(H)| \le k|V(H)|$ for all $H \subset G[V_i]$.
They proved that this condition is sufficient.

\begin{oldthm}[\candh, \cite{caro12}] \label{thm:alldeg}
	Let $G$ be an $r$-uniform hypergraph and $k \ge 0$ an integer. Then $H$ has an orientation such
	that for each vertex $v$ some $1 \le i \le r$ satisfies $d_i(v) \le k$ if and
	only if $V$ can be partitioned into $r$ sets $V_1, \dotsc, V_r$ such that for
	each $j$ and each $U \subset V_j$ at most $k|U|$ edges $e$ satisfy $e \subset U$.
\end{oldthm} 

For degrees of multiple vertices, the first case $k=0$ is already non-obvious.
When can an $r$-uniform hypergraph $G$ be given an orientation such that for any $p$-set of vertices
$A$ there is \emph{some} $p$-set $I \subset [r]$ such that $d_I(A) = 0$? Such an orientation would
partition $V\tup{p}$ into $\binom{r}{p}$ sets $W_I, \, I \in [r]\tup{p},$ where each
$A \in V\tup{p}$ is put into $W_I$ with $d_I(A) = 0$. For any $I$ and any edge $e$
the $p$-set of vertices that are in positions labeled by $I$ in $D(e)$ must not belong
to $W_I$. So there are no edges whose $p$-sets would all belong to a single $W_I$, giving a necessary
condition. \candh{} asked if, similarly to the case $p = 1$, this condition is
sufficient.

\begin{oldqn}[\candh, \cite{caro12}] \label{qn:orient-part}
    Let $1 \le p \le r$ be integers and $G$ an $r$-uniform hypergraph. Suppose $V\tup{p}$
    is partitioned into sets $W_1, \dotsc, W_R$, where $R = \binom{r}{p}$, in such a way
    that for any $1 \le i \le R$ there are no edges $e$ such that $e\tup{p} \subset W_i$.
    Must $G$ have an orientation such that for any $p$-set of vertices $A$ there is some
    $p$-set $I \subset [r]$ such that $d_I(A) = 0$?
\end{oldqn}

Our main aim of the paper is to show that the answer to Question~\ref{qn:orient-part} is positive for
$r$ much larger than $p$, but negative in general.

It will turn out in this paper that the following notion is crucial to understanding
Question~\ref{qn:orient-part}.  We wish to determine if a function
$f \colon \sset{n}{p} \to \sset{n}{p}$ must `fix an intersection' in the sense explained by the next definition.

\begin{defn}
    Let $1 \le p \le n$ be integers and $f \colon \sset{n}{p} \to \sset{n}{p}$ a function. We say
    that $f$ \emph{fixes an intersection} if there are distinct $x, y \in \sset{n}{p}$ such that
	$|f(x) \cap f(y)| = |x \cap y|$.

    Moreover, if every nonconstant function $f \colon \sset{n}{p} \to \sset{n}{p}$ fixes an intersection
    then we say that $\sset{n}{p}$ has the \emph{fixed intersection property}.
\end{defn}

It is not difficult to see that if $p \ge 2$ then $\sset{2p}{p}$ does not have the fixed intersection
property. This can be demonstrated by choosing $y = [p]$, $\bar{y} = [2p] \sm [p] = \{p+1, \dotsc, 2p\}$
and defining $f \colon \sset{2p}{p} \to \sset{2p}{p}$ by
$$
    f(x) =
    \begin{cases}
        y & \text{ if } x = y \text{ or } x = \bar{y} \\
        \bar{y} & \text{ otherwise}.
    \end{cases}
$$
This $f$ is nonconstant and does not fix an intersection.

We conjecture that $n = 2p$ (for $p \ge 2$) is actually the only exceptional case.

\begin{conj} \label{conj:main}
    Let $1 \le p \le n$ be integers with $n \neq 2p$. Then $\sset{n}{p}$
    has the fixed intersection property.
\end{conj}

In Section \ref{sec:hypergraphs} we show that, if Conjecture \ref{conj:main} is true for given
$p$ and $n$, then Question \ref{qn:orient-part} can be answered positively 
(for the same choice of $p$ with $r = n$). In Section \ref{sec:ramsey} we show Conjecture
\ref{conj:main} is true when $r$ is much larger than $p$.

\begin{restatable}{thm}{conjpart} \label{thm:conjpart}
	For every integer $p \ge 1$ there is some $n_0$ such that if $n \ge n_0$ then
    $\sset{n}{p}$ has the fixed intersection property.
\end{restatable}

\begin{cor}
	The answer to Question \ref{qn:orient-part} is positive if $r$ is sufficiently large, given $p$.
\end{cor}

What if $[n]\tup{p}$ does not have the fixed intersection property? The simplest such
case is $n = 4, p = 2$. The main part of this paper is devoted to showing how the trivial
failure of the fixed intersection property of $[4]\tup{2}$ can be `lifted' to a failure
of Question \ref{qn:orient-part} for the case $r = 4, p = 2$. Our work relies on a new version
of the `amalgamation' technique in sparse (structural) Ramsey theory.

\begin{restatable}{thm}{nogen} \label{thm:nogen}
	There is a $4$-uniform hypergraph $H$ satisfying:
	\begin{enumerate}[(a)]
		\item \label{it:nilhak} for every orientation of $H$ there is some pair of vertices
			$u,v$ such that $d_I(u,v) > 0$ for all $I \in [4]\tup{2}$;
		\item \label{it:sixpart} there is a partition of $V\tup{2}$ into six
			sets $V_1, \dotsc, V_6$ such that $e\tup{2}$ is not contained
            in a single $V_j$ for any edge $e$.
	\end{enumerate} 

	In particular, the answer to Question \ref{qn:orient-part} is negative for $r = 4, p = 2$.
\end{restatable}

One such hypergraph is constructed in Section \ref{sec:sparseramsey}. In Section \ref{sec:ordereddegrees}
we give a brief analysis of the `ordered' notion of degrees. Finally, in Section \ref{sec:openproblems}
we suggest some open problems.

%\setcounter{tocdepth}{1}
%\tableofcontents

\section{Degree theorems for hypergraphs} \label{sec:hypergraphs}

In this section we prove Theorem \ref{thm:hakforpsets} and give a positive answer to
Question~\ref{qn:orient-part} for the choices of $r$ and $p$ such that $[r]\tup{p}$
has the fixed intersection property.

We start with an easy proof of Theorem \ref{thm:hakforpsets}.

\hakforpsets*

\begin{proof}
	Construct a bipartite graph $H$ with vertex classes $X = E(G)$ and $Y = V(G)\tup{p}$.
	Join $e \in X$ and $A \in Y$ by an edge if $A \subset e$. In other words,
	we join each edge of $G$ to its $p$-subsets. Given an orientation $D(G)$ on $G$, let
	$S$ be the set of edges in $H$ of the form $eA$ where $A$ is the set of vertices
    in the first $p$ positions of $D(e)$, that is, $A = \{v_1, \dotsc, v_p\}$ where
	$D(e) = (v_1, \dotsc, v_r)$. Note that $S$ covers every element of $X$ exactly once and,
	conversely, every subset of $E(H)$ that covers every element of $X$ exactly once is
    given by some choice of $D(G)$.  Moreover, each $A \in Y$ is covered by $S$ exactly
    $d_{[p]}(A)$ times.

	Therefore by Hall's marriage theorem, $G$ has an orientation such that all 
	$A \in V(G)\tup{p}$ satisfy $d_{[p]}(A) \le k$ if and only if 
	$|W| \le k|\Gamma(W)|$ for all $W \subset X$, where $\Gamma(W) \subset Y$ is the
	neighbourhood of $W$ in $H$. This can be restated as $|W| \le k|U|$ for all $W \subset X$
	and $U \subset Y$ with the property that $\Gamma(W) \subset U$. Having chosen $U$, the
	largest $W$ such that $\Gamma(W) \subset U$ is $\{e\in E(G) \colon e\tup{p}
	\subset U\}$. Hence, this statement is equivalent to $k|U| \ge |\{e \in E(G) \colon 
	e\tup{p} \subset U\}|$ for all $U \subset V(G)\tup{p}$.
\end{proof}

Now we switch our attention to Question \ref{qn:orient-part}. We are trying to give
an $r$-uniform hypergraph $G$ an orientation such that for each $p$-set of vertices
\emph{some} $I$-degree is zero. In particular, if $e$ is an edge and $A \subset e$ is a
$p$-set, we require $d_{I_A}(A) = 0$ for some $I_A$. So we need a tool that would
give $e$ an orientation when for each $p$-set $A \subset e$ there is a $p$-set
$I_A \subset [r]$ such that the vertices $A$ cannot take exactly the positions labeled
by $I_A$. Of course, no such orientation exists
if all sets $I_A$ are equal (some $p$ vertices must be in positions labeled
by that set). However, if the sets $I_A$ are not all equal then it turns out that
such an orientation exists provided $[r]\tup{p}$ has the fixed intersection property.

\begin{lem} \label{lem:orient-edge}
    Let $1 \le p \le r$ be integers such that $[r]\tup{p}$ has the fixed intersection
    property and let $e$ be an $r$-set. For each $p$-set $A \subset e$ choose a $p$-set
    $I_A \subset [r]$. If the chosen $p$-sets $I_A$ are not all equal then
    the elements of $e$ can be ordered in such a way that the positions labeled by
    $I_A$ are not taken exactly by the elements of $A$ for any $p$-set $A \subset e$.
\end{lem}

\begin{proof}
    Without loss of generality assume that $e = [r]$. Then we have a nonconstant function
    $\sset{r}{p} \to \sset{r}{p}$ that sends $A \in \sset{r}{p}$ to $I_A$.

    Let $\sigma$ be a permutation of $[r]$, chosen uniformly at random. For each $p$-set
    $A \subset [r]$, $\sigma(A) = I_A$ with probability $1/\binom{p}{r}$. Therefore
    \begin{align*}
        \Pb\left[ \sigma(A) = I_A \text{ for some $p$-set } A \subset [r] \right]
            &\le \sum_{A \in \sset{r}{p}} \Pb\left[ \sigma(A) = I_A \right] \\
            &= 1
    \end{align*}
    with equality if and only if the events $\sigma(A) = I_A$ are pairwise disjoint.

    Since $\sset{r}{p}$ has the fixed intersection property,
    $|I_{B} \cap I_{C}| = |B \cap C|$ for some distinct $p$-sets $B, C \subset [r]$.
    Hence there is a permutation $\tau$ of $[r]$ such that $\tau(B) = I_{B}$
    and $\tau(C) = I_C$, so the events $\sigma(B) = I_B$ and $\sigma(C) = I_C$
    are not disjoint, and so the above inequality is strict. Therefore there is
    a choice of $\sigma$ such that $\sigma(A) \neq I_A$ for any $A$. The required
    ordering of $e$ is $(\sigma(1), \dotsc, \sigma(r))$ for this choice of $\sigma$.
\end{proof}

This immediately gives a positive answer to Question \ref{qn:orient-part} for the case
when $\sset{r}{p}$ has the fixed intersection property.

\begin{prop} \label{prop:partthenorient}
    Let $1 \le p \le r$ be integers such that $\sset{r}{p}$ has the fixed intersection
    property and let $G$ be an $r$-uniform hypergraph. Suppose $V\tup{p}$
    is partitioned into sets $W_1, \dotsc, W_R$, where $R = \binom{r}{p}$, in such a way
    that for any $1 \le i \le R$ there are no edges $e$ such that $e\tup{p} \subset W_i$.
    Then $G$ has an orientation such that for any $p$-set of vertices $A$ there is some
    $p$-set $I \subset [r]$ such that $d_I(A) = 0$.
\end{prop}

\begin{proof}
    Relabel the sets $W_1, \dotsc, W_R$ with $p$-subsets of $[r]$ (instead of integers $1, \dotsc, R$).
    Given an edge $e$ of $G$, associate to each $p$-set $A \subset e$ the $p$-set $I_A \subset [r]$
    where $A \in W_{I_A}$. By assumption not all of these $p$-sets are equal, so Lemma \ref{lem:orient-edge}
    can be applied to get an orientation $D(e)$ that does not contribute to $d_{I_A}(A)$ for
    any $A$. Do this for every edge to get an orientation of $G$. For any $p$-set of vertices $A$,
    $I_A$ does not depend on the choice of $e$, and $d_{I_A}(A) = 0$. 
\end{proof}

Can the requirement for $\sset{r}{p}$ to have the fixed intersection property be removed?
As we shall see in Section \ref{sec:sparseramsey}, proposition would not be true for \mbox{$r = 4, p = 2$}
which is in fact the smallest case when $\sset{r}{p}$ does not have the fixed intersection
property. Therefore some condition on $r$ and $p$ must be imposed. It might be interesting to
know if some weaker condition would be enough.

On the other hand, this result can be strengthened to match the form of Theorem \ref{thm:alldeg}.
We fix an integer $k \ge 0$ and ask if $G$ can be given an orientation such that for every
$p$-set of vertices some $I$-degree is at most $k$. With Theorem \ref{thm:hakforpsets},
it is once again it is easy to find a sufficient condition.

\begin{prop}
    Let $k \ge 0$, $1 \le p \le r$ be integers and $G$ an $r$-uniform hypergraph. Suppose that
    $G$ has an orientation such that for each $p$-set of vertices $A$ there is some $p$-set
    $I \subset [r]$ such that $d_I(A) \le k$. Then $V\tup{p}$ can be partitioned into $R
    = \binom{r}{p}$ sets $W_1, \dotsc, W_R$ such that for each $j$ and each $U \subset W_j$
    at most $k|U|$ edges $e \in E(G)$ satisfy $e\tup{p} \subset U$.
\end{prop}

\begin{proof}
    Partition $V\tup{p}$ into $R$ disjoint sets $W_I, \, I \in \sset{r}{p},$ by putting
    each \mbox{$A \in V\tup{p}$} into $W_I$ such that $d_I(A) \le k$. For each $I$ apply
    Theorem \ref{thm:hakforpsets} to the induced subgraph $G[W_I]$.

    (To make the definition
    of induced subgraphs precise, given $W \subset V(G)\tup{p}$, $G[W]$ is the $r$-uniform
    hypergraph with vertex set $V(G)$ and edge set \linebreak \mbox{$\{e \in E(G) \colon e\tup{p} \subset W\}$}.
    Informally, we keep only those edges whose $p$-subsets are in $W$.)
\end{proof}

The argument used for the case $k = 0$ together with Theorem \ref{thm:hakforpsets} is enough
to prove sufficiency when $\sset{r}{p}$ has the fixed intersection property.

\begin{prop} \label{prop:partthenorient-genk}
    Let $k \ge 0, 1 \le p \le r$ be integers such that $\sset{r}{p}$ has the fixed intersection
    property and let $G$ be an $r$-uniform hypergraph. Suppose $V\tup{p}$ is partitioned
    into sets $W_1, \dotsc, W_R$, where $R = \binom{r}{p}$, in such a way that for each $j$
    and each $U \subset W_j$ at most $k|U|$ edges $e \in E(G)$ satisfy $e\tup{p} \subset U$.
    Then $G$ has an orientation such that for each $p$-set of vertices $A$ there is some $p$-set
    $I \subset [r]$ such that $d_I(A) \le k$.
\end{prop}

\begin{proof}
    Relabel the sets $W_1, \dotsc, W_R$ with $p$-subsets of $[r]$ (instead of integers $1, \dotsc, R$).
    By Theorem \ref{thm:hakforpsets} we can give orienations to the edges of the induced subgraphs
    $G[W_I]$, $I \in \sset{r}{p}$, in such a way that, only counting these edges,
    $d_I(A) \le k$ whenever $A \in W_I$. Now, same as in the proof of Proposition
    \ref{prop:partthenorient}, Lemma \ref{lem:orient-edge} allows us to give orientation
    to the remaining edges of $G$ without increasing $d_I(A)$ for any $A \in W_I$.
\end{proof}

\section{Fixed intersection property of $\sset{n}{p}$ for $n$ large} \label{sec:ramsey}

Here we prove Theorem \ref{thm:conjpart} which says that $\sset{n}{p}$ has the fixed
intersection property for pairs $(n, p)$ with $n$ much larger than $p$.
We start by extending the definitions given in the introduction to slightly greater
generality.

\begin{defn}
    Let $p \ge 1$ be an integer and $S, T$ sets. We say that a function $f \colon
    S\tup{p} \to T\tup{p}$ \emph{fixes an intersection} if there are distinct
    $x, y \in S\tup{p}$ such that $|f(x) \cap f(y)| = |x \cap y|$.

    Moreover, we say that $S\tup{p}$ has the \emph{fixed intersection property} if
    every nonconstant function $f \colon S\tup{p} \to S\tup{p}$ fixes an intersection.
\end{defn}

Here is a technical lemma.

\begin{lem} \label{lem:technical}
	Let $p \ge 1$ be an integer and $S \subset \N$. If $f \colon S\tup{p} \to \N\tup{p}$
	is a nonconstant function that is constant on $M\tup{p}$ for some $M \subset S$ with
	$|M| \ge 2p-1$ then $f$ fixes an intersection.
\end{lem}

\begin{proof}
	Suppose for contradiction $f$ does not fix an intersection. Note that $S\sm~M \neq \es$
    as $f$ is nonconstant.

    Write $M_0 = M$ and $i_0 = \min(S\sm M_0)$. Suppose $x = x' \cup \{i_0\}$ where
    $x' \in M_0\tup{p-1}$. Then either $f(x) = f(y)$ or
    $|f(x) \cap f(y)| \le p-1$ for all $y \in M_0\tup{p}$. In the latter case, pick
    $y \in M_0\tup{p}$ such that $|y \cap x'| = |f(x) \cap f(y)|$. This is possible because
	$|M_0| \ge 2p-1$. Then $|f(x) \cap f(y)| = |x \cap y|$, so $f$ fixes an intersection,
	contradicting our assumption. Therefore for any $x$ as defined above and any 
    $y \in M_0\tup{p}$ we have $f(x) = f(y)$. This means that $f$ is constant on $M_1\tup{p}$
    where $M_1 = M_0 \cup \{i_0\}$.

    Now write $i_1 = \min(S\sm M_1)$ and repeat the argument to conclude that $f$ is constant
    on $M_2\tup{p}$ where $M_2 = M_1 \cup \{i_1\}$. Repeat it again to show that $f$ is
    constant on $M_n\tup{p}$ for every $n$ where
    $M_0 \subset M_1 \subset \dotsb$ are nested sets that cover $S$. But this
    implies $f$ is constant on $S\tup{p}$. Contradiction.
\end{proof}

We use this lemma to get the result for infinite domains.

\begin{lem} \label{lem:innat}
    For any positive integer $p$, $\N\tup{p}$ has the fixed intersection property.
\end{lem}

\begin{proof}
	Use induction on $p$. Clearly, theorem holds for $p = 1$ so suppose $p \ge 2$. 

	Let $f \colon \N\tup{p} \to \N\tup{p}$ be a nonconstant function.  Without losing generality
	assume that $f([p]) = [p]$. Now $c(x) = f(x) \cap [p]$ for $x \in \N\tup{p}$ defines a
	finite colouring of $\N\tup{p}$ where colours are subsets of $[p]$. By Ramsey's theorem
	there is an infinite set $M \subset \N$ such that $c$ is monochromatic on $M\tup{p}$.
	Say, $c(x) = C \subset [p]$ for all $x \in M\tup{p}$. 

	If $C = \es$ pick any $x \in (M \sm [p])\tup{p}$. Then $|x \cap [p]| = 0 = |f(x) \cap f([p])|$,
    so $f$ fixes an intersection.

	If $C = [p]$ then $f(x) = [p]$ for all $x \in M\tup{p}$ but this cannot happen
	by Lemma \ref{lem:technical}.

    Now suppose $C \neq \es, [p]$  and write $s = |C|$. For fixed $z \in M\tup{s}$ define
    $$
        g \colon (M \sm z)\tup{p-s} \to (\N \sm C)\tup{p-s}
    $$
    by $g(x) = f(x \cup z) \sm C$ for all
    $x \in (M \sm z)\tup{p-s}$.  As $f$ is not constant on $p$-sets of $M$ by Lemma \ref{lem:technical},
    there is a choice of $z$ for which $g$ is not constant. Then the induction hypothesis implies that
    $|g(x) \cap g(y)| = |x \cap y|$ for some distinct $x, y \in (M \sm z)\tup{p-s}$ and so
    $|f(x \cup z) \cap f(y \cup z)| = |g(x) \cap g(y)| + s = |x \cap y| + |z| = |(x \cup z) \cap (y \cup z)|$.
\end{proof}

A compactness argument extracts the result for finite domains.

\begin{cor} \label{cor:infrang}
	Let $p \ge 1$ be an integer. Then there is some $n\ge p+1$ such that every nonconstant
    function $f \colon [n]\tup{p} \to \N\tup{p}$ fixes an intersection.
\end{cor}

\begin{proof}
	Suppose not. Then for any integer $n \ge p+1$ there is a nonconstant function
	$f_n \colon [n]\tup{p} \to \N\tup{p}$ that does not fix an intersection. By reordering $\N$,
	if necessary, we can assume that $f_n([s]\tup{p}) \subset \left[p \binom{s}{p} \right]\tup{p}$
	for each $n \ge p+1$ and each $s \le n$.

	We define $f \colon \N\tup{p} \to \N\tup{p}$ as follows. For any $s$, there are only finitely
	many choices for $f_n$ on $[s]\tup{p}$ so, in particular, infinitely many $f_n$ agree on
	$[p]$. Define $f([p])$ to be $f_n([p])$ for any such $n$. Among these $f_n$, infinitely
	many agree on $[p+1]\tup{p}$. 	Define $f$ on $[p+1]\tup{p}$ to be the same as any such
	$f_n$. Among these $f_n$, infinitely many agree on $[p+2]\tup{p}$ and we define	$f$ on
	$[p+2]\tup{p}$ to agree with these $f_n$. Continue, at each step restricting to an infinite
	subsequence of the $f_n$. We get $f$ that, for every $N$, agrees with some $f_{n(N)}$ on
    $[N]\tup{p}$. In particular, $f$ is nonconstant since $f_{n(2p-1)}$ is nonconstant on $[2p-1]\tup{p}$
	by Lemma \ref{lem:technical}. By Lemma \ref{lem:innat}, $|f(x) \cap f(y)| = |x \cap y|$ for
	some distinct $x, y \in \N\tup{p}$. But $x, y \in [N]\tup{p}$ for some $N$ so $|f_{n(N)}(x)
	\cap f_{n(N)}(y)| = |x \cap y|$, hence $f_{n(N)}$ fixes an intersection. Contradiction.
\end{proof}

We get Theorem \ref{thm:conjpart} as an immediate corollary.

\conjpart*
\vspace{-24px}\qed % hack hack hack

\section{Sparse Ramsey type counterexample} \label{sec:sparseramsey}

\subsection{Overview}

In light of Conjecture \ref{conj:main} and Proposition \ref{prop:partthenorient}, we seek $p \ge 2$
such that Question \ref{qn:orient-part} has a negative answer when $r = 2p$. It turns out that $p = 2$ 
works. We recall the exact statement that we will prove.

\nogen*

The choice $p = 2$ allows us to consider the elements of $V(H)\tup{2}$ as edges and non-edges
of a graph on vertices $V(H)$. With this idea in mind we will deduce Theorem \ref{thm:nogen} from
a statement about graphs rather than hypergraphs.

\begin{defn}
	\label{defn:special}
	Take six colours and partition them into three pairs. Call two colours \emph{opposite}
	if they are in the same pair.

	Suppose $A$ is a $4$-clique of a graph whose edges are coloured with these six colours.
	We say that $A$ is \emph{special} if there is a pair of opposite colours $c_1, c_2$ such that
	two independent edges of $A$ have colour $c_1$ and the remaining four edges have colour $c_2$.
\end{defn}

\begin{lem}
	\label{lem:main}
	There is a graph $G$ satisfying:
	\begin{enumerate}[(a)]
		\item \label{it:nomono} edges of $G$ can be coloured with six colours without
			forming a monochromatic $4$-clique;
		\item \label{it:monoorspecial} whenever the edges of $G$ are coloured with six
			colours there is a monochromatic $4$-clique or a special $4$-clique.
	\end{enumerate}
\end{lem}

\begin{proof}[Proof of Theorem \ref{thm:nogen} (assuming Lemma \ref{lem:main})]
    Let $G$ be a graph as given by Lemma \ref{lem:main}. Form a $4$-uniform hypergraph $H$
    on the vertices of $G$ by taking the $4$-cliques of $G$ as its edges (so $V(H) = V(G)$ and
    $E(H) = \{ A \in V(G)\tup{4} \colon A\tup{2} \subset E(G) \}$).

    Property (\ref{it:nomono}) of Lemma \ref{lem:main} for $G$ directly implies property
    (\ref{it:sixpart}) of Theorem \ref{thm:nogen} for $H$.
    
    Suppose for contradiction $H$ does not have property (\ref{it:nilhak}) of
    Theorem \ref{thm:nogen}, that is $H$ has an orientation $D(H)$ such that for any
    pair of vertices $u,v$ there is some $I \in \sset{4}{2}$ such that $d_I(u,v) = 0$. 
    In particular, there is an edge colouring $c \colon E(G) \to \sset{4}{2}$
    such that $d_{c(e)}(e) = 0$ for all $e \in E(G)$. As is evident from the proof
    of Lemma \ref{lem:orient-edge} (or by a simple check), the induced colouring
    $V(A)\tup{2} \to \sset{4}{2}$ on any $4$-clique $A \subset G$ fixes an intersection.
    
    Partition the set of colours $\sset{4}{2}$ into three pairs, each consisting of two
    disjoint $2$-sets. Let these be the pairs of opposite colours. Property (\ref{it:monoorspecial})
    of Lemma \ref{lem:main} for $G$ implies that some $4$-clique $A \subset G$ is
    monochromatic or special under $c$. If we identify the vertices of $A$ with $1, 2, 3, 4$
    then on $A$ $c$ induces a mapping $c_A \colon \sset{4}{2} \to \sset{4}{2}$. If $A$ is
    monochromatic then this mapping is
    $$
    c_A(x) = \{1, 2\} \text{ for all } x \in \sset{4}{2}
    $$
    under some permutation of $1, 2, 3, 4$. If $A$ is special then under some permutation
    the mapping is
    $$
        c_A(x) =
        \begin{cases}
            \{1,2\} &\text{if } x = \{1,2\} \text{ or } x = \{3,4\} \\
            \{3,4\} &\text{otherwise}.
        \end{cases}
    $$
    In both cases $c_A$ does not fix an intersection, giving a contradiction.

    This finishes the proof.
\end{proof}

The proof of Lemma \ref{lem:main} is based on a new amalgamation method in sparse Ramsey theory.
Amalgamation (also known in literature as \emph{partite construction}) was introduced by \nandr{}
in \cite{nesetril79}. There have been many more applications of this technique. See, for example,
\cite{nesetril77, nesetril87} or \cite{leader-sparseramsey} for general introduction.

Here is a brief outline of the method. Suppose we wish to find a graph $G$ that is sparse (e.g., does
not contain a copy of some fixed graph)
but has the property that every colouring of its \emph{vertices} with $k$ colours contains a monochromatic
copy of some fixed graph $H$. We start with a sparse $t$-partite graph $\hat{G}$ with the property that if
its vertices are coloured so that every vertex class is monochromatic then there must be a monochromatic
copy of $H$. Then a larger $t$-partite graph $G$ is formed by taking many copies of $\hat{G}$ and
glueing them together in a specific way (this process usually consists of several steps). The goal
is to ensure that $G$ is as sparse as $\hat{G}$ and that any vertex colouring of $G$ contains a copy
of $\hat{G}$ with monochromatic vertex classes (and therefore contains a monochromatic copy of $H$).

A variant of this technique is useful in finding sparse $G$ with the property that every colouring
of its \emph{edges} with $k$ colours contains a monochromatic copy of $H$. The main difference is
that we start with a sparse $t$-partite graph with the property that if its edges are coloured so
that every pair of vertex classes spans a monochromatic bipartite graph then there must be a
monochromatic copy of $H$.

Our result is of slightly different type. The sparsity condition (\ref{it:nomono}) in Lemma \ref{lem:main}
is standard but the Ramsey property (\ref{it:monoorspecial}) is not. The difference is that instead of
requiring a monochromatic structure that is smaller than the forbidden structure we require a structure
of the same size but not necessarily monochromatic.

To adjust for this, instead of working with $t$-partite graphs we use graphs that are `almost'
$t$-partite. More specifically, vertices are partitioned into $t$ sets that span sparse rather than
empty subgraphs.

\subsection{Preparing for proof of Lemma \ref{lem:main}}

Let $G$ and $H$ be graphs and let $k$ be a positive integer. We say that $G$ is
\emph{$k$-edge-Ramsey for} $H$ if every colouring of the edges of $G$  with $k$ colours
contains a monochromatic copy (not necessarily induced) of $H$. Similarly, $G$ is
\emph{$k$-vertex-Ramsey for} $H$ if this holds for every colouring of the vertices of $G$
with $k$ colours. Note that $G$ containing $H$ as a subgraph is equivalent both to $G$ being
$1$-edge-Ramsey for $H$ and to $G$ being $1$-vertex-Ramsey for $H$.

Any $t$-partite graph (where $t$ is a fixed positive integer) is considered as having
its vertex classes listed in a fixed order. More formally, any such $G$ is a pair
$((V_i)_{i=1}^t, E)$ where $(V_i)_{i=1}^t$ is the list of vertex classes and $E$ is the set of
edges of $G$. We write $(V_i)$ to mean $(V_i)_{i=1}^t$ when there is no danger of confusion.
The order of vertex classes is respected when standard operations on graphs are carried out.
For example, 
\begin{itemize}
	\item if $G^\alpha = ((V_i^\alpha)_{i=1}^t, E^\alpha)$ are a family of $t$-partite graphs labeled by
        $\alpha$ then their union is
		$$
			\bigcup_{\alpha} G^\alpha = \left(\Big(\bigcup_\alpha V_i^\alpha\Big)_{i=1}^t,
			\bigcup_\alpha E^\alpha\right),
		$$
        provided $V_i^\alpha \cap V_j^\beta = \es$ when $i \neq j$. So the \ith{i} vertex class
        of the union of graphs is the union of \ith{i} vertex classes;
	\item if $G$ and $H$ are $t$-partite graphs then $G$ contains $H$ as a
        subgraph if there is an injective graph homomorphism $H \hookrightarrow G$ that maps each vertex
        class of $H$ into the corresponding vertex class of $G$;
	\begin{figure}[ht]
	\centering
	\includegraphics{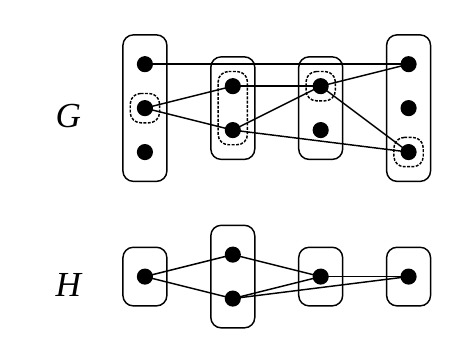}
	\caption{$G$ contains $H$.}
	\end{figure}
	\item if $G$ and $H$ are $t$-partite graphs then $G$ is a copy of $H$ if
        there is a graph isomorphism $H \hookrightarrow G$ that maps each vertex
        class of $H$ onto the corresponding vertex class of $G$.

\end{itemize}

Let $k$ be a fixed positive integer. Given a $t$-partite graph $G = ((V_i), E)$ for any
$j = 1, \dotsc, t$ we can construct a $t$-partite graph $\A_j(G)$, called the
\emph{\ith{j} amalgam of} $G$. The main property of $\A_j(G)$ is that whenever its vertices
are coloured with $k$ colours there must be a copy of $G$ with all vertices in $V_j$ of
the same colour.

We construct $\A_j(G)$ as follows. Let $W_j$ be any set with $|W_j| = k|V_j|$. For any
$A \in W_j^{(|V_j|)}$ let $G^A$ be a copy of $G$ whose \ith{j} vertex class is exactly the set $A$.
Choose these copies in such a way that they would intersect only at the \ith{j} vertex class
(so $G^A \cap G^B = A \cap B$ for any $A \neq B$).
Now define $\A_j(G) = \bigcup_{A \in W_j^{(|V_j|)}} G^A$ to be the union of all these
copies of $G$.
\begin{figure}[ht]
\centering
\includegraphics{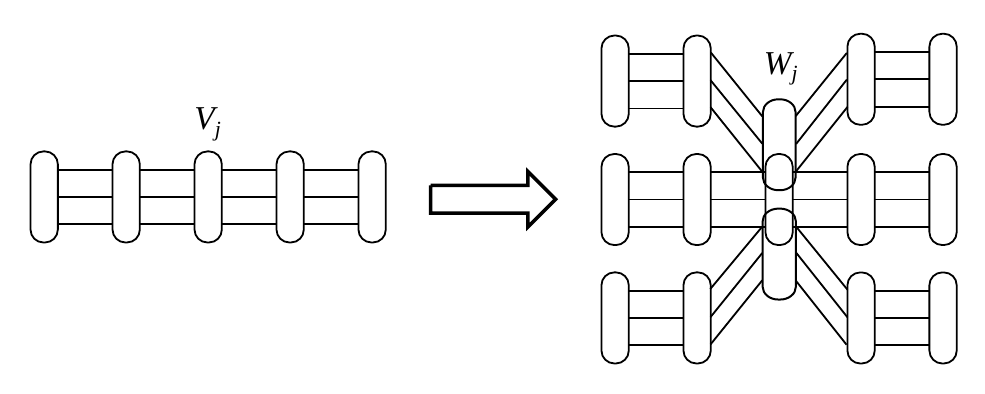}
\caption{Construction of the \ith{j} amalgam.}
\end{figure}

The pigeonhole principle implies that whenever the vertices of $\A_j(G)$ are coloured with
$k$ colours, some set $A \in W_j^{(|V_j|)}$ is monochromatic. Then $G^A$ is a copy
of $G$ in $\A_j(G)$ with the \ith{j} vertex class monochromatic. Furthermore,
it is not difficult to see that if for some $s$ there are no $s$-cliques in $G$ then there
are none in $\A_j(G)$ either. We will need a slightly stronger result.

\begin{prop} \label{prop:amal-noter}
	Let $c : G \to X$ be an edge colouring without a monochromatic copy of $K_s$. Then there is
	an edge colouring $\A_j(c) : \A_j(G) \to X$ without a monochromatic copy of $K_s$.
	In other words, if for fixed $k$ and $s$ $G$ is not $k$-edge-Ramsey for $K_s$ then neither
	is $\A_j(G)$.
\end{prop}

\begin{proof}
	Every edge $e$ of $\A_j(G)$ is an edge of $G^A$ for a unique $A \in W_j^{(|V_j|)}$.
	Let $\A_j(c)(e) = c(e)$ where on the right hand side $e$ is considered to be an edge
	of $G \cong G^A$. Suppose for contradiction $\A_j(c)$ produces a monochromatic copy
	of $K_s$ on vertices $v_1, \dotsc, v_s$. At most one of them is in $W_j$ as it is an
	independent set so we can assume $v_1 \not\in W_j$. Now $v_1 \in G^A$ for a unique $A$
	and we must have $v_r \not\in G^A$ for some $2 \le r \le s$ since otherwise the vertices
	$v_1, \dotsc, v_s$ would span a monochromatic copy of $K_s$ in $G^A \cong G$ under the
	colouring induced by $c$, contradicting our assumption on $c$. But then $v_1 v_r$ is not
	an edge of $\A_j(G)$ which is absurd.
\end{proof}

\begin{lem} \label{lem:v-vs-e}
	Let $b, s$ and $k$ be positive integers and suppose $H$ is a graph that is not
	$b$-edge-Ramsey for $K_s$. Then there is a graph $G$ such that
	\begin{enumerate}[(i)]
		\item $G$ is $k$-vertex-Ramsey for $H$, and
		\item $G$ is not $b$-edge-Ramsey for $K_s$.
	\end{enumerate}
\end{lem}

\begin{proof}
	Define $t = k|H|$ and take a $t$-partite graph $G_0$ that is a union of vertex
	disjoint copies of $H$ such that for any choice of $|H|$ vertex classes of
	$G_0$ some copy of $H$ intersects all of them. Note that if the vertices of $G_0$ are
	coloured with $k$ colours and each vertex class is monochromatic
	then there must be a monochromatic copy of $H$.
    Since $H$ is not $b$-edge-Ramsey for $K_s$, neither is $G_0$.

	Take $G = \A_t(\A_{t-1}(\dotsb\A_1(G_0)\dotsb))$. By Proposition \ref{prop:amal-noter},
	$G$ is not $b$-edge-Ramsey for $K_s$. However, suppose the vertices of $G$ are coloured
	with $k$ colours. From previous observations we know that $G$ contains a copy $G'$ of \linebreak
	$\A_{t-1}(\A_{t-2} (\dotsb\A_1(G_0)\dotsb))$ with the \ith{t} vertex class monochromatic.
	Moreover, $G'$ is $k$-vertex-coloured so contains a copy $G''$ of
	$\A_{t-2}(\A_{t-3}(\dotsb\A_1(G_0)\dotsb))$ with the $(t-1)^\text{\tiny st}$ vertex class
	monochromatic (so in fact, both the \ith{t} and the $(t-1)^\text{\tiny st}$ vertex classes
	of $G''$ are monochromatic). Repeat this argument to deduce that $G$ contains a copy of
	$G_0$ with all vertex classes monochromatic. Our construction of $G_0$ now ensures that
	there is a monochromatic copy of $H$.
\end{proof}

A special case is particularly useful.

\begin{cor} \label{cor:v}
	Let $s$ and $k$ be positive integers and suppose $H$ is a graph that contains no copy
	of $K_s$. Then there is a graph $G$ such that
	\begin{enumerate}[(i)]
		\item $G$ is $k$-vertex-Ramsey for H, and
		\item $G$ does not contain $K_s$ as a subgraph.
	\end{enumerate}
\end{cor}
\begin{proof}
	This is Lemma \ref{lem:v-vs-e} with $b = 1$.
\end{proof}

A slightly more involved application of the method produces the following result. It is folklore,
but for completeness we outline its proof as described in \cite{leader-sparseramsey}.

\begin{lem} \label{lem:e-vs-ks}
	Let $s$ and $k$ be positive integers and suppose $H$ is a graph that does not contain a
	copy of $K_s$. Then there is a graph $G$ such that
	\begin{enumerate}[(i)]
		\item $G$ is $k$-edge-Ramsey for $H$, and
		\item $G$ does not contain $K_s$ as a subgraph.
	\end{enumerate}
\end{lem}

\begin{proof}[Proof]
	Again, we will start with a $t$-partite graph $G_0$ but this time choose $t = R_k(|H|)$,
	that is $t$ is the smallest integer such that whenever the edges of $K_t$ are
	coloured with $k$ colours we are guaranteed to get a monochromatic copy of $K_{|H|}$.
	Construct $G_0$ by taking vertex disjoint copies of $H$ in such a way that whenever $|H|$
	vertex classes of $G_0$ are taken there is a copy of $H$ that intersects them
	all. Note that if the edges of $G_0$ are coloured with $k$ colours and if every pair
	of vertex classes spans a monochromatic bipartite graph then there is a monochromatic
	copy of $H$.

	We will apply a slighty modified process of amalgamation to $G_0$. Given a $t$-partite
	graph $G_\alpha = ((V_i), E)$ and integers $1 \le i < j \le t$, define a $t$-partite
	graph $\A_{i,j}^\ast(G_\alpha)$ as follows. Let $G_{i,j} = G_\alpha[V_i \cup V_j]$ be
	the bipartite graph induced by $G_\alpha$ on vertex classes $V_i$ and $V_j$. Let
	$G_{i,j}^\ast$ be a bipartite graph on vertex classes $X$ and $Y$ such that whenever the
	edges of $G_{i,j}^\ast$ are coloured with $k$ colours, one of its induced subgraphs
	isomorphic to $G_{i,j}$ (with $V_i \subset X$ and $V_j \subset Y$) is monochromatic. It is
	not obvious that $G_{i,j}^\ast$ exists but we defer the explanation to the end of the proof.
	For every induced subgraph $A \subset G_{i,j}^\ast$ that is isomorphic to $G_{i,j}$
	(with $V_i \subset X$ and $V_j \subset Y$), take a copy $G^A$ of $G_\alpha$ making sure
	that $G^A[V_i, V_j] = A$ and that the copies are disjoint everywhere else. Finally, define
	$\A_{i,j}^\ast(G_\alpha) = \bigcup_A G^A$. Then $\A_{i,j}^\ast(G_\alpha)$ does not contain
	$K_s$ as a subgraph (taking induced copies of $G_{i,j}$ in the previous step was crucial)
	and whenever the edges of $\A_{i,j}^\ast(G_\alpha)$ are coloured with $k$ colours
	we are guaranteed to get a copy of $G_\alpha$ with the edges joining vertex classes
	$V_i$ and $V_j$ having the same colour.

	List all pairs $(i,j)$, $1 \le i < j \le t$, in some order $(i_1, j_1), \dotsc,
	(i_T, j_T)$ where $T = \binom{t}{2}$. Define
	$G = \A_{i_1, j_1}^\ast \left(\A_{i_2, j_2}^\ast\left(\dotsb \A_{i_T, j_T}^\ast\left(G_0\right)
	\dotsb \right)\right).$
	If the edges of $G$ are coloured with $k$ colours then $G$ is bound to have a copy of $G_0$
    every pair of vertex classes spanning a monochromatic bipartite subgraph, and hence a monochromatic
    copy of $H$. Moreover, $G$ does not contain $K_s$ as a subgraph.

	It remains to prove the existence of $G_{i,j}^\ast$. For this we use the well known
	Hales--Jewett theorem (see \cite{hales63}). Let the vertex classes be $X = V_i^n$ and
	$Y = V_j^n$ where $n$ is a large integer. Join $(x_1, \dotsc, x_n)$ and $(y_1, \dotsc, y_n)$
	by an edge if $x_1 y_1, \dotsc, x_n y_n$ are edges in $G_{i,j}$. So if $F$ is the set of
	edges of $G_{i,j}$ then $G_{i,j}^\ast$ has edges $F^n$. By the Hales--Jewett theorem, for
	sufficiently large $n$ every colouring of $F^n$ with $k$ colours has a monochromatic
	combinatorial line. It is not difficult to see that a combinatorial line in $F^n$ corresponds
    to induced copy of $G_{i,j}$.
\end{proof}

We proceed by proving a proposition which we will use to extend the technique to non-partite
graphs. Before stating it we introduce a simple construction. Given graphs $G$ and $H$, define the
\emph{pair graph} of $G$ and $H$ to be the graph obtained by taking vertex disjoint copies of $G$ and
$H$ and joining each vertex in the copy of $G$ to each vertex in the copy of $H$ by an edge. Denote
this graph by $G \pg H$ and call the edges joining the copy of $G$ to the copy of $H$
\emph{intermediate}. 

\begin{prop} \label{prop:technical}
	Let $k$ and $s$ be positive integers and suppose $G_0$ and $H_0$ are graphs that do not
	contain $K_s$. Then there are graphs $G$ and $H$ such that
	\begin{enumerate}[(i)]
		\item $G$ and $H$ do not contain $K_s$, and
		\item whenever the intermediate edges of $G \pg H$ are coloured with $k$ colours,
			there is a copy of $G_0 \pg H_0$ with $G_0 \subset G$, $H_0 \subset H$ and
			all intermediate edges monochromatic.
	\end{enumerate}
\end{prop}

\begin{proof}
	By Corollary \ref{cor:v}, there is a graph $H$ that is $k$-vertex-Ramsey for $H_0$ but
	contains no $s$-cliques. There is also a graph $G$ that is $k^{|H|}$-vertex-Ramsey
	for $G_0$ but contains no $s$-cliques.

	Let $c$ be any colouring of the intermediate edges of $G \pg H$ with colours $[k]$. It
	induces a vertex colouring $c_G : G \to [k]^{V(H)}$ given by $c_G(x) = (c(xy))_{y\in V(H)}$.
	By our choice of $G$, some copy $G_0 \subset G$ is monochromatic under $c_G$. In other words, the
	colour of an intermediate edge of $G_0 \pg H$ only depends on its endpoint in $H$ and not on
	the one in $G_0$. Now define a vertex colouring $c_H : H \to [k]$ by letting each vertex of
	$H$ have the colour of any edge joining it to $G_0$. Some copy $H_0 \subset H$ is
	monochromatic under $c_H$ which means exactly that all intermediate edges of $G_0 \pg H_0$
	have the same colour.
\end{proof}

\subsection{Proof of Lemma \ref{lem:main}} \label{subsec:mainproof}

Let $\hat{G}$ be a fixed graph and suppose that for each vertex $v \in \hat{G}$ a graph $F_v$
is chosen, and write $\F = (F_v)_{v\in V(\hat{G})}$ for the collection of chosen graphs.
Define the \emph{$\F$-blowup of $\hat{G}$}, denoted by $\hat{G}(\F)$, as follows.
\begin{itemize}
    \item Each vertex $v$ of $\hat{G}$ is replaced by a copy of $F_v$ and these copies are
        pairwise vertex disjoint.
    \item For every edge $uv$ of $\hat{G}$, each vertex of $F_u$ is joined to each vertex of
        $F_v$ by an edge. Call such edges \emph{$uv$-intermediate}.
    \item For every nonedge $uv$ of $\hat{G}$ there are no edges between $F_u$ and $F_v$.
\end{itemize}
Note that the previously introduced pair graph is a special case of a blowup: $G \pg H$ is
the same as $K_2(G, H)$.
\begin{figure}[ht]
	\centering
	\includegraphics{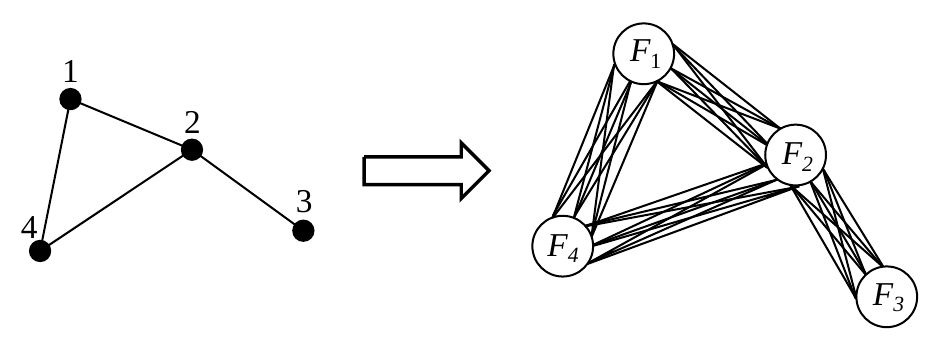}
	\caption{Construction of blowup.}
\end{figure}

Lemma \ref{lem:e-vs-ks} implies the existence of a graph $F$ that is $6$-edge-Ramsey for
$K_3$ but contains no copies of $K_4$. Set $\F_0 = (F)_{v \in V(\hat{G})}$, that is, a copy
of $F$ is chosen for each vertex $v \in \hat{G}$. Enumerate the edges of $\hat{G}$ as
$e_1, \dotsc, e_n$ and consider them one by one in this order. Suppose $e_i = u_i v_i$ is
considered and we have a list of graphs $\F_{i-1} = (F_v)_{v \in V(\hat{G})}$. By
Proposition \ref{prop:technical}, there are graphs $F'$ and $F''$ such that
\begin{enumerate}[(i)]
	\item $F'$ and $F''$ contain no copies of $K_4$, and
	\item whenever the intermediate edges of $F' \pg F''$ are coloured with six colours,
		we can find copies $F_{u_i} \subset F'$ and
	$F_{v_i} \subset F''$ with all intermediate edges of $F_{u_i} \pg F_{v_i}$ of the same colour.
\end{enumerate}

Replace $F_{u_i}$ by $F'$ and $F_{v_i}$ by $F''$ to obtain a new list $\F_i$. Do this for all edges
in turn and use the final list $\F_n$ to form the blowup $G = \hat{G}(\F_n)$.

Take any colouring $c$ of the edges of $G$ with six colours. By the choice of $G$, it contains
a copy of $\hat{G}(\F_{n-1})$ with the $e_n$-intermediate edges monochromatic. Furthermore, the
edges of this copy are also coloured with six colours and so it contains a copy of
$\hat{G}(\F_{n-2})$ with the $e_{n-1}$-intermediate edges monochromatic (so in fact, both the
$e_n$-intermediate and the $e_{n-1}$-intermediate edges are monochromatic). Repeat this argument
to conclude that there is a copy of $\hat{G}(\F_0)$ with the $e$-intermediate edges monochromatic
for each edge $e \in \hat{G}$. Moreover, the edges of the copy of $F$ which corresponds to a
vertex $v \in \hat{G}$ are coloured with six colours so it contains a monochromatic
$K_3$. Having this in mind, $c$ induces a colouring $\hat{c}$ of both the edges and the vertices
of $\hat{G}$ (so $\hat{c}$ is a total colouring) with six colours where
\begin{align*}
	\hat{c}(v) &= \text{colour of a monochromatic } K_3 \text{ in the corresponding copy of } F
		& \text{ if } v \in V(\hat{G}) \\
	\hat{c}(e) &= \text{colour of all } e \text{-intermediate edges in the copy of } \hat{G}(\F_0)
		& \text{ if } e \in E(\hat{G}).
\end{align*}
Notice that if $\hat{c}$ contains
\begin{enumerate}[\hspace{12pt}(a)]
	\item \label{item:e-k4} a $K_4$ with all edges of the same colour, or
	\item \label{item:ve} a vertex and an incident edge of the same colour
\end{enumerate}
then $c$ contains a monochromatic $K_4$. Moreover, if $\hat{c}$ contains
\begin{enumerate}[\hspace{12pt}(a)] \setcounter{enumi}{2}
	\item \label{item:eopep} an edge $e$ both of whose endpoints have the colour opposite
		to the colour of $e$
\end{enumerate}
then $c$ contains a special $K_4$.

Conversely, suppose that $\hat{c}$ is a total colouring of $\hat{G}$ with six colours. It induces
an edge colouring $c$ of $G$ with six colours where
$$
	c(e) =
	\begin{cases}
		\hat{c}(v) & \text{if } e \text{ is an edge of the graph which corresponds to a vertex }
			v \in \hat{G} \\
		\hat{c}(e') & \text{if } e \text{ is an } e' \text{-intermediate edge.}
	\end{cases}
$$

A monochromatic copy of $K_4$ appears in $c$ if and only if $\hat{c}$ contains (\ref{item:e-k4}) or
(\ref{item:ve}). Putting this together with earlier considerations reduces the initial problem to 
finding a graph $\hat{G}$ that admits a total colouring with six colours with no (\ref{item:e-k4})
and (\ref{item:ve}) but whose every total colouring with six colours produces (\ref{item:e-k4}),
(\ref{item:ve}) or (\ref{item:eopep}).

Start with a graph $H$ that is $4$-edge-Ramsey for $K_4$ but not $5$-edge-Ramsey for $K_4$.
Lemma \ref{lem:v-vs-e} enables us to choose $\hat{G}$ that is $3$-vertex-Ramsey for $H$ but not
$5$-edge-Ramsey for $K_4$. If we use five colours to colour the edges of $\hat{G}$ avoiding
a monochromatic $K_4$ and use another colour for all the vertices, we get a total colouring
without (\ref{item:e-k4}) and (\ref{item:ve}). On the other hand, given a total colouring of
$\hat{G}$ with six colours, first consider the colours in each pair of opposite colours as
being the same, thereby reducing the number of colours to three. By the choice of $\hat{G}$, there
is a copy of $H$ with monochromatic vertex set. This means that in the original total colouring
the vertices of this copy of $H$ are coloured using only one pair of opposite colours, say, red
and blue. If some edge of this copy of $H$ is red or blue then we get (\ref{item:ve}) or
(\ref{item:eopep}). Otherwise, only four colours are used for these edges so we have
(\ref{item:e-k4}).

This concludes the proof of Lemma \ref{lem:main}. \qed

\section{Ordered degrees} \label{sec:ordereddegrees}

In this section we define and briefly examine the notion of `ordered' degrees for sets of
multiple vertices. Let $D(G)$ be an orientation of an $r$-uniform hypergraph $G$. Given
a pair of vertices $u,v$, we can define $\dd_{12}(u,v)$ to be the number of edges $e$
such that $u$ is in the first position of $D(e)$ and $v$ is in the second. For example,
if $E(G) = \left\{(4,5,1), (4,1,3), (1,4,2)\right\}$ then $\dd_{12}(1,4) = 1$.

More generally, for an \emph{ordered} $p$-tuple of distinct vertices $A = (v_1, \dotsc, v_p)$
and an \emph{ordered} $p$-tuple $I = (i_1, \dotsc, i_p) \subset [r]$, the \emph{ordered $I$-degree} of $A$,
denoted by $\dd_I(A)$, is the number of edges $e$ such that the elements of $D(e)$ in
positions labeled by $I$ are the vertices $v_1, \dotsc, v_p$ in \emph{this order}. More
formally, $\dd_I(A)$ is the number of edges $e$ such that if we write $D(e) = (x_1, \dotsc, x_r)$
then $x_{i_1} = v_1, \dotsc, x_{i_p} = v_p$.

In this section by a \emph{$p$-tuple} we will always mean an ordered $p$-set without
repeated elements. For any set $S$ we denote by $S\otup{p}$ the family of all
$p$-tuples with elements from $S$. For example,
$[3]\otup{2} = \left\{ (1,2), (1,3), (2,1), (2,3), (3,1), (3,2) \right\}$.

Following the spirit of Theorems \ref{thm:hakimi}, \ref{thm:hakforhyp} and
Question~\ref{qn:genhak}, we can ask when an $r$-uniform hypergraph $G$ can be given an
orientation such that $\dd_{(1,2,\dotsc,p)}(A) \le k$ for all $p$-tuples of vertices $A$,
where $k \ge 0$ is a fixed integer.

It is easy to find a necessary condition: for any collection of $p$-sets 
$U \subset V\tup{p}$ there can be no more than $k p! |U|$ edges $e$ such that $e\tup{p}
\subset U$. Indeed, every such edge contributes to the sum
$$
\sum_{A \in U} \sum_{\substack{A^\ast \text{ an} \\ \text{ordering of } A}}
\dd_{(1,2,\dotsc,p)}(A^\ast)
$$
by $1$ and this sum does not exceed $k p! |U|$.

The proof of sufficiency is almost identical to Theorem \ref{thm:hakforpsets}.

\begin{thm} \label{thm:hakforptuples}
    Fix integers $k \ge 0, 1 \le p \le r$ and let $G$ be an $r$-uniform hypergraph.
    Suppose that for any $U \subset V\tup{p}$ there are at most $k p! |U|$ edges
    $e$ such that $e\tup{p} \subset U$. Then $G$ has an orientation such that
    $\dd_{(1,2,\dotsc,p)}(A) \le k$ for every $p$-tuple $A \subset V$.
\end{thm}

\begin{proof}
    Construct a bipartite graph $H$ with vertex classes $X = E(G)$ and
    $Y = V(G)\otup{p}$. Join $e \in X$ and $A \in Y$ by an edge if $e$ contains
    all elements of $A$.

    Suppose $S \subset X$ and let $\Gamma(S) \subset Y$ be the neighbourhood
    of $S$ in $H$. All $p!$ permutations of any element of $\Gamma(S)$ are in
    $\Gamma(S)$ so by assumption $|S| \le k|\Gamma(S)|$. By Hall's marriage
    theorem there is a $1$-to-$k$ matching from $X$ to $Y$. Orient each edge
    $e \in E(G) = X$ in such a way that the initial $p$-positions of $D(e)$
    are the $p$-tuple to which $e$ is joined in this matching. This produces
    an orientation of $G$ with the required property.
\end{proof}

A version of Question \ref{qn:orient-part} can be asked for ordered degrees.
When does an $r$-uniform hypergraph $G$ have an orientation such that for each
$p$-tuple of vertices $A$ there is a $p$-tuple $I \subset [r]$ such that
$\dd_I(A) = 0$? If $p = 1$ then this is covered by Theorem \ref{thm:alldeg}
so let us assume $p \ge 2$.

In contrast to the notion of `unordered' degrees, it turns out that every
$G$ has such an orientation. In fact, this can be achieved by a simple explicit
construction.

\begin{thm}
    Let $2 \le p \le r$ be integers and $G$ an $r$-uniform hypergraph. Then $G$ has an
    orientation such that for each $p$-tuple of vertices $A$ there is a $p$-tuple
    $I \subset [r]$ such that $\dd_I(A) = 0$.
\end{thm}

\begin{proof}
    Without loss of generality assume that $V(G) = [n]$. For each edge $e$, order
    its vertices in the increasing order.

    Let $A = (a_1, \dotsc, a_p)$ be a $p$-tuple of vertices. If $a_1 < \dotsb < a_p$
    then $\dd_{(p,p-1,\dotsc,1)}(A) = 0$. Otherwise, $\dd_{(1,2,\dotsc,p)}(A) = 0$.
\end{proof}

\section{Open problems} \label{sec:openproblems}

Question \ref{qn:orient-part} has a negative answer if $(r, p) = (4, 2)$ which is the
smallest pair for which there is a nonconstant $f : [r]\tup{p} \to [r]\tup{p}$ that does
not fix an intersection, and our proof went by constructing
a graph that is not $6$-edge-Ramsey for $K_4$ but whose every edge-colouring with
colours $[4]\tup{2}$ induces a colouring on a $K_4$ that does not fix an intersection.
It might be interesting to know if Question \ref{qn:orient-part} has a negative answer
exactly when the conclusion of Conjecture \ref{conj:main} is false, and perhaps this could
be done by adapting the same construction for hypergraphs.

\begin{qn}
    Let $1 \le p \le r$ be integers such that $\sset{r}{p}$ does not have the fixed intersection
    property. Must there be an $r$-uniform hypergraph $G$ satisfying:

    \begin{enumerate}[(a)]
        \item $V\tup{p}$ can be partitioned into $R = \binom{r}{p}$ sets $W_1, \dotsc, W_R$ in
            such a way that for any $1 \le i \le R$ there are no edges $e$ such that
            $e\tup{p} \subset W_i$;

        \item whenever $G$ is given an orientation there is some $p$-set of vertices $A$ such that
            $d_I(A) > 0$ for all $p$-sets $I \subset [r]$?
    \end{enumerate}
\end{qn}

Using brute-force computer search we were able to check Conjecture \ref{conj:main} for
cases $r \le 5$. Our main progress in the general case is achieved by showing that pairs
$(r, p)$ with $r$ much larger than $p$ have
the stronger property that every nonconstant $f \colon [r]\tup{p} \to \N\tup{p}$ fixes
an intersection (Corollary \ref{cor:infrang}). 

\begin{qn} \label{it:infrang}
    For what choices of $r$ and $p$  does every nonconstant function
    $f \colon \sset{r}{p} \to \N\tup{p}$ fix an intersection?
\end{qn}

The interesting choices of $r$ and $p$ are when $r$ is slightly larger than $2p$. Computer
search reveals that this question is satisfied by $(r, p) = (8, 3)$ but not by $(r, p) = (7, 3)$.
Therefore Conjecture \ref{conj:main} cannot be proved solely by answering this question. However,
if $f \colon [7]\tup{3} \to [m]\tup{3}$ does not fix an intersection then $m$ must be large.
In fact, the smallest such $m$ is $31$. Curiously, every such $f$ (even if $m > 31$) is injective.
Therefore it might be that the additional injectivity condition, would not change the answer.

\begin{qn} \label{it:infranginj}
    For what choices of $r$ and $p$ does every injection $f \colon \sset{r}{p} \to \N\tup{p}$
    fix an intersection?
\end{qn}

In Conjecture \ref{cor:infrang} we deal with functions $\sset{r}{p} \to \sset{r}{p}$. In
this setting the analogue of Question \ref{it:infranginj} is as follows.

\begin{qn} \label{it:finrangbij}
    For what choices of $r$ and $p$ does every bijection $f \colon \sset{r}{p} \to \sset{r}{p}$
    fix an intersection?
\end{qn}

There is a simple counting argument that shows Question \ref{it:finrangbij} is satisfied by
$r = \Omega(p^2)$. Suppose $f \colon [r]\tup{p} \to [r]\tup{p}$ is a bijection and
assume $f([p]) = [p]$ for convenience. There are exactly $\binom{r-p}{p}$ elements of
$[r]\tup{p}$ that do not intersect $[p]$ and $\binom{r}{p} - \binom{r-p}{p}$ elements
that intersect $[p]$. So if there is no $x \in [r]\tup{p}$ such that
$|f(x) \cap [p]| = |x \cap [p]| = 0$ then $\binom{r}{p} > 2\binom{r-p}{p}$ which is
equivalent to
$\left(1 + p/(r-p)\right) \left(1 + p/(r-p-1)\right) \dotsb \left( 1 + p/(r-1)\right) > 2$.
But then $\exp\left(p/(r-p) + \dotsb + p/(r-1)\right) > 2$. If $r \ge cp^2$ then the left
hand side is at most $e^{1/(c-1)}$ which is not greater than $2$ provided $c$ is large enough.

\section*{Acknowledgements}

We would like to thank Imre Leader for his invaluable comments regarding the presentation of this
article.

\bibliographystyle{elsarticle-num}
%\bibliographystyle{abbrv}
%\bibliographystyle{plain}
%\bibliography{bib/paper}

\end{document}